\documentclass[a4paper]{amsart}

\usepackage{amsthm, amssymb}
\usepackage{cite}
\usepackage[dvipdfmx]{graphicx}

\theoremstyle{definition}
\newtheorem{theorem}{Theorem}[section]
\newtheorem{lemma}[theorem]{Lemma}
\newtheorem{proposition}[theorem]{Proposition}
\newtheorem{observation}[theorem]{Observation}

\theoremstyle{definition}
\newtheorem{definition}[theorem]{Definition}

\newtheorem{pclaim}[theorem]{Claim}

\newtheorem*{ac}{Acknowledgements}

\theoremstyle{remark}

\newenvironment{rmenum}{
\begin{enumerate}

}
{\end{enumerate}}

\newcommand{\parcut}[2]{\delta_{#1}(#2)}

\newcommand{\reverse}[1]{#1^{-1}}
\newcommand{\nonneg}{\ge 0}
\newcommand{\nonnegz}{Z_{\nonneg}}

\newcommand{\acset}[2]{\mathcal{S}^{#2}(#1)}

\newcommand{\nacset}[2]{\mathcal{T}^{#2}(#1)}
\newcommand{\sqrset}[1]{\mathcal{A}(#1)}

\newcommand{\sdigraph}{\mathcal{G}}
\newcommand{\stcomp}[1]{\mathcal{C}(#1)}
\newcommand{\st}{\preceq}

\newcommand{\stset}[1]{\mathcal{O}(#1)}
\newcommand{\parstset}[2]{\mathcal{O}^{#2}(#1)}

\newcommand{\se}[2]{\sigma(#1; #2)}
\newcommand{\tse}[2]{\tilde{\sigma}(#1; #2)}

\title[Bidirected Critical Graphs I]{Constructive Characterization for Bidirected Analogue of Critical Graphs I: Principal Classes of Radials and Semiradials}
\author{Nanao Kita}
\address{Tokyo University of Science 
2641 Yamazaki, Noda, Chiba, Japan 278-0022}
\email{kita@rs.tus.ac.jp}

\begin{document}

\begin{abstract}
This paper is the first from  serial papers that provide  constructive characterizations for classes of bidirected graphs known as radials and semiradials. 
In this paper, we provide  constructive characterizations for five principle classes of radials and semiradials to be used 
 for characterizing general radials and semiradials. 
A bidirected graph is a graph in which each end of each edge has a sign $+$ or $-$. 
Bidirected graphs are a common generalization of digraphs and signed graphs. 
We define a new concept of radials as a generalization of a classical concept in matching theory,  critical graphs. 
Radials are also a generalization of a class of digraphs known as flowgraphs. 
We also define semiradials, which are a relaxed concept of radials. 
We further define special classes of radials and semiradials, 
that is,   absolute semiradials,  strong  and almost strong radials, 
  linear semiradials, and sublinear radials. 
 We  provide constructive characterizations for these five classes of bidirected graphs. 
 Our serial papers are a part of a series of works that establish the strong component decomposition for bidirected graphs. 
\end{abstract} 

\maketitle

\section{Introduction} 

Bidirected graphs are a common generalization of digraphs and signed graphs.  
Bidirected graphs were first proposed in 1970 by Edmonds and Johnson~\cite{Edmonds70matching:a} 
to provide a unified integer linear programming formulation for various combinatorial optimization problems. 
A bidirected graph is a graph in which each end of each edge has a sign $+$ or $-$. 
A digraph is a bidirected graph in which two ends of each edge have distinct signs.  
A signed graph is a graph in which each edge has a single sign. 
Therefore, this is a special bidirected graph in which two ends of each edge have the same sign.

In this paper, we define a new class of bidirected graphs, {\em radials}, 
that is a common generalization of critical graphs and flowgraphs. 
Critical graphs are a classical concept in matching theory~\cite{lp1986}. 
In the context of $1$-matchings or $1$-factors, critical graphs are also called factor-critical graphs.    
There is an easy correspondence between   a signed graph and a graph endowed with a set of edges. 
Under this correspondence,  critical graphs are equivalent to signed graphs 
in which every vertex can reach a specified vertex along  directed trails starting and ending with signs $-$ and $+$, respectively.  
In contrast, a directed graph  is  called a flowgraph 
if every vertex can reach a specified vertex along  directed paths or, equivalently, along  directed trails. 
We define the concept of radials  
as the bidirected graphs in which every vertex can reach a specified vertex 
along a directed trail that starts and ends with signs $-$ and $+$, respectively.  
We also define a relaxed concept of radials, {\em semiradials}, 
as bidirected graphs in which every vertex can reach a specified vertex along  directed trails starting with $-$. 

Constructive characterizations for classes of graphs can be  strong tools,  because they are useful in inductive proofs. 
Ear decompositions are a general term that refers to a type of inductive construction methods of graphs. 
In ear decompositions, a graph is constructed from a single vertex or circuit by repeatedly adding paths or circuits. 
Lov\'asz gave a constructive characterization for factor-critical graphs in terms of ear decompositions~\cite{lovasz1972a, lp1986}. 
Strongly connected digraphs can also be characterized using a directed version of ear decompositions~\cite{schrijver2003}.

Our aim  is to provide a constructive characterization for the class of radials. 
For attaining this aim,  a  constructive characterization of semiradials is also required. 
Thus, we aim to characterize radials and semiradials. 
In this paper, 
we define five principal classes of radials and semiradials, that is, 
{\em absolute} semiradials, {\em strong}  and {\em almost strong} radials, {\em linear} semiradials, and {\em sublinear} radials, 
 and provide constructive characterizations for these five classes. 
In the sequel of this paper~\cite{kitasearii}, 
we provide  constructive characterizations for general radials and semiradials 
using the five principal classes that we give in this paper. 
Our two papers, this paper and its sequel, 
are a part of a series of works that establish the strong component decomposition for bidirected graphs. 
See also Kita~\cite{kita2017bidirected}.

We outline the five principal classes of radials and semiradials as follows. 
We first define absolute semiradials as semiradials in which every vertex can reach a specified vertex along directed trails starting with $-$ and $+$, 
and then provide a characterization of this class in a form that is similar to ear decompositions. 
We then define two subclasses of absolute semiradials, that is, strong  and almost strong radials.  
We give constructive characterizations for  these two classes using the characterization of absolute semiradials.  
On the other hand, 
we define linear semiradials as semiradials in which no vertex can reach the specified vertex along ditrails starting with $+$.  
We also define sublinear radials as a similar counterpart concept to linear semiradials. 
It is revealed that these two classes have structures that is similar to flowgraphs, 
and we provide a characterization for these classes using the strong component decomposition of digraphs 
and the constructive characterization of strongly connected digraphs.

The remaining part of this paper is constructed as follows. 
Sections~\ref{sec:notation} to \ref{sec:flowgraph} are devoted to  preliminaries. 
The basic notation is explained in Section~\ref{sec:notation}. 
The strong component decomposition and strongly connected digraphs are explained in Section~\ref{sec:digraph}. 
In Section~\ref{sec:digraphic}, the relationship between digraphs and bidirected graphs is explained. 
In Sections~\ref{sec:flowgraph} and \ref{sec:critical},  flowgraphs and critical graphs are explained. 
  New concepts and results are introduced from Section~\ref{sec:radial} onward. 
In Section~\ref{sec:radial}, we define new concepts of radials and semiradials. 
In Sections~\ref{sec:absolute}, \ref{sec:strong}, and \ref{sec:astr}, 
we provide the definitions and constructive characterizations 
for absolute semiradials, strong radials, and almost strong radials, respectively. 
In Section~\ref{sec:linear}, 
we provide the definitions and constructive characterizations 
for linear semiradials and sublinear radials.

\section{Notation} \label{sec:notation} 

\subsection{Graphs} \label{sec:notation:graph}

We mostly follow Schrijver~\cite{schrijver2003} for basic notation and definitions. 
In this section, we list exceptions or nonstandard definitions that we use. 
We denote the set of nonnegative integers by $\nonnegz$. 
Let $G$ be an (undirected) graph. 
We denote the vertex and edge sets of $G$ by $V(G)$ and $E(G)$, respectively.  
We consider mutigraphs. That is, loops and parallel edges may exist. 
For $u,v\in V(G)$, $uv$ denotes an edge whose ends are $u$ and $v$. 
As usual, a singleton $\{x\}$ is often denoted by $x$.

Let $X\subseteq V(G)$. 
We denote the cut of $X$, that is, the set of edges that join $X$ and $V(G)\setminus X$, by $\parcut{G}{X}$. 
The subgraph of $G$ induced by $X$ is denoted by $G[X]$. 
We often denote $G[V(G)\setminus X]$ by $G -X$. 
We sometimes treat a graph as the set of its vertices.  

Let $H$ be a supergraph of $G$, and let $F\subseteq E(H)$. 
The graphs obtained by adding $F$ to $G$ and deleting $F$ from $G$ are denoted by $G + F$ and $G-F$, respectively. 
For two subgraphs $G_1$ and $G_2$ of $H$, 
the addition of $G_1$ and $G_2$ is denoted by $G_1 + G_2$. 

Assume that $G$ is connected, and $v\in V(G)$ is a cut vertex of $G$, that is, 
$G-v$ has more than one connected components. 
For each connected component of $C$ of $G-v$, 
we call $G[V(C)\cup \{v\}]$ the {\em block} of $G$  over $v$.

Now, let $G$ be a graph. 
Let $s,t\in V(G)$. 
A {\em walk} from $s$ to $t$ is a sequence $(w_1, \ldots, w_k)$, where $k\ge 1$, such that 
 $k$ is odd, $w_1 = s$,  $w_k = t$, and $w_i \in V(G)$ holds for each odd $i\in \{1,\ldots, k\}$, 
whereas $w_i$ is an edge in $E(G)$ that joins $w_{i-1}$ and $w_{i+1}$ for each even $i\in \{1,\ldots, k\}$.  
We call $s$ and $t$ the {\em ends} of this walk. 
If $s$ and $t$ are the same vertex $r$, then we say that a walk is {\em closed} over $r$. 
A {\em trail} is a walk in which no edge is contained more than once. 
A {\em path} is a trail in which no vertex is contained more than once.

Let $W$ be a walk $(w_1, \ldots, w_k)$, where $k\ge 1$. 
We denote by $\reverse{W}$  the walk $(w_k, \ldots, w_1)$. 
Let $W'$ be another walk $(w_k, \ldots, w_l)$, where $k \le l$. 
Then, $W + W'$ denotes the concatenation of $W$ and $W'$, that is, the walk $(w_1, \ldots, w_l)$.  
Note that this operation is not commutative. 

A vertex or edge $v$ of $G$ is said to be a vertex or edge of $W$ if $v$ is contained in $W$. 
We denote the sets of vertices and edges of $W$ by $V(W)$ and $E(W)$.  
We sometimes treat a walk $W$ as the graph whose vertex and edge sets are $V(W)$ and $E(W)$. 
For a subgraph $H$ of $G$, 
the addition $H + W$ of $H$ and $W$ denotes the addition of $H$ and the graph $W$.

The {\em terms} of $W$ are  $k$ variables $t_1, \ldots, t_k$ that are ordered from $1$ to $k$.  
We say that a term $t$ of $W$ {\em denotes} $w_i$ if $t$ is the $i$-th term of $W$. 
For each odd or even $i \in \{ 1, \ldots, k \}$, we call $t_i$ a {\em vertex} or {\em edge term}, respectively. 
For vertex terms $t_i$ and $t_j$ with $i \le j$, 
we denote the subwalk $(w_i, \ldots, w_j)$ of $W$ by $t_iWt_j$.  
For simplicity, we often denote $t_i$ by $w_i$ if the meaning is obvious from the context. 
If $i=1$ or $j=k$, we often denote $t_iWt_j$ by $w_1Wt_j$ or $t_iWw_k$, respectively. 
Furthermore, 
we often denote $t_iWt_j$ by $w_iWw_j$.

\subsection{Bidirected Graphs}

A bidirected graph is a graph in which each end of each edge has a sign $+$ or $-$. 
A precise definition is as follows. 
Let $G$ be a graph. 
Let $\partial_+$ and $\partial_-$ be mappings $E(G) \rightarrow  2^{V(G)}$ that satisfy the following conditions 
for each $e \in E(G)$ with (possibly identical) ends $u$ and $v$. 
\begin{rmenum} 
\item For each $\alpha\in\{+, -\}$, $\partial_{\alpha}(e) \subseteq \{u, v\}$ holds.  
\item $\partial_+(e) \cup \partial_-(e) = \{ u, v \}$.  
\item  If $e$ is not a loop, then $\partial_+(e) \cap \partial_-(e) = \emptyset$. 
\end{rmenum} 
Then, the graph $G$ endowed with $\partial_+$ and $\partial_-$ is called a {\em bidirected graph}.  
We say that a sign of $u$ over $e$ is $\alpha$ if $ u \in \partial_\alpha(e)$ holds for $\alpha\in\{+, -\}$. 
If $e$ is not a $(+, -)$-loop and $u\in \partial_\alpha(e)$ holds,  then we denote 
the sign of $u$ over $e$ by $\se{u}{e}$. 
If $\partial_{\alpha}(e) = \{u ,v\}$ for some $\alpha\in\{+, -\}$, then $e$ is called an $(\alpha, \alpha)$-edge. 
In contrast, if $\partial_\alpha(e) \cap \{u, v\} \neq\emptyset$ holds for each $\alpha \in \{+, -\}$,  then $e$ is called a $(+, -)$- 
or $(-, +)$-edge. 

Bidirected graphs are a common generalization of digraphs and signed graphs.
A digraph is a special bidirected graph, in that, 
a digraph is a bidirected graph in which every edge is a $(+, -)$-edge. 
A signed graph is a bidirected graph in which every edge is a $(+, +)$- or $(-, -)$-edge. 
The notation for (undirected) graphs that are introduced in Section~\ref{sec:notation:graph} 
can be naturally defined for bidirected graphs and directed or signed graphs.

We define directed  walks for bidirected graphs. 
Let $G$ be a bidirected graph, and let $W$ be a walk in $G$ of the form $(w_1, \ldots, w_k)$, where $k\ge 1$. 
Let $t_1,\ldots, t_k$ be the terms of $W$.  
We call $W$ a {\em directed walk} or {\em diwalk} if 
there exists a mapping $\tilde{\sigma}$ that satisfies the following conditions:  
\begin{rmenum} 
\item 
For each even $i\in\{1,\ldots, k\}$,  
$\tse{t_{i-1}}{t_i} = \se{w_{i-1}}{w_i}$ and $\tse{t_{i+1}}{t_i} = \se{w_{i+1}}{w_i}$  if $w_i$ is not a $(+, -)$-loop; 
\item 
if $w_i$ is a $(+, -)$-loop, then 
$\tse{t_{i-1}}{t_i}$ and $\tse{t_{i+1}}{t_i}$ are mutually distinct signs $+$ or $-$. 
\item 
For each odd $i\in\{1,\ldots, k\}\setminus \{1, k\}$, 
$\tse{t_i}{t_{i-1}}$ and $\tse{t_i}{t_{i+1}}$ are mutually distinct signs. 
\end{rmenum} 
It is easily observed that if $W$ is a diwalk, then the mapping that satisfies this condition  uniquely exists. 
We denote  $\tilde\sigma$ by $\sigma$  under the assumption that $W$ is a diwalk. 

Assume that $W$ is a diwalk in the following. 
We denote $\se{t_1}{t_2}$ and $\se{t_k}{t_k-1}$ by $\se{t_1}{W}$ and $\se{t_k}{W}$, respectively, if $k\ge 3$. 
If $\se{t_1}{W} = \alpha$ and $\se{t_k}{W} = \beta$ for $\alpha, \beta \in \{+, -\}$, then $W$ is said to be an $(\alpha, \beta)$-ditrail. 
We define the trivial ditrail with $k=1$ to be a $(+, -)$- and $(-, +)$-ditrail. 
For any $\beta\in\{+, -\}$, an $(\alpha, \beta)$-ditrail is called an $\alpha$-ditrail. 
We often denote $\se{t_1}{W}$ and $\se{t_k}{W}$  
by $\se{w_1}{W}$ and $\se{w_k}{W}$,  
and call these values the signs of $w_1$ and $w_2$ over $W$ 
 if the meaning is obvious from the context.

A {\em directed trail} or {\em ditrail} is a diwalk in which no edge is contained more than once. 
A {\em directed path} or {\em dipath} is a ditrail in which no vertex is contained more than once.

We now define directed ear or diear in bidirected graphs. 
Let $X\subseteq V(G)$. 
A  diwalk $W$ is a {\em diear} relative to $X$  if the ends of $W$ are contained in $X$, 
its edges are disjoint from the edges of $G[X]$, and 
\begin{rmenum} 
\item \label{item:ear:simple} $W$ is a ditrail, or 
\item \label{item:ear:baloon} $W$ is of the form $(v, e, w_3, \ldots, w_{k-2}, e, v)$, where $k\ge 7$, 
such that $e\in \parcut{G}{X}$ holds and  $(w_3, \ldots, w_{k-2})$ is a closed ditrail that does not contain $e$. 
\end{rmenum} 
A ditrail of the form \ref{item:ear:simple} is called a {\em simple} diear. 
A ditrail of the form \ref{item:ear:baloon} is called a {\em scoop} diear,  
for which $e$  is called the {\em grip}. 
A scoop diear $W$ is called an $\alpha$-scoop diear if the sign of $v$ over $e$ is $\alpha$.

\section{Digraphs} \label{sec:digraph}

Let $G$ be a digraph. 
Two vertices $u$ and $v$ of $G$ are {\em strongly connected} 
if $G$ has dipaths from $u$ to $v$ and from $v$ to $u$. 
A digraph is {\em strongly connected} if every two vertices are strongly connected. 
A {\em strongly connected component} or {\em strong component} 
is a maximal strongly connected subgraph. 
We denote the set of strong components of $G$ by $\stcomp{G}$. 
It is easily observed from these definitions that the following properties hold for strong components: 
\begin{rmenum} 
\item 
If $C, D\in\stcomp{G}$ are distinct, then $C$ and $D$ are disjoint.  
\item 
$\bigcup \{ V(C): C\in\stcomp{G} \} = V(G)$. 
\end{rmenum}

\begin{definition} 
Let $G$ be a digraph. 
Define a binary relation $\st$ over $\stcomp{G}$ as follows: 
\begin{rmenum} 
\item 
For $D_1, D_2\in\stcomp{G}$, 
let $D_1 \st D_2$ if $G$ has an arc from a vertex in $D_1$ to a vertex in $D_2$. 
\item 
For $D_1, D_2\in\stcomp{G}$, 
let $D_1 \st D_2$ 
if there exists $C_1,\ldots, C_k\in\stcomp{G}$, where $k\ge 1$, such that  
$C_1 = D_1$, $C_k = D_k$, and $C_i\st C_{i+1}$ for every $i\in \{1,\ldots, k\}\setminus \{k\}$. 
\end{rmenum} 
\end{definition} 

The next proposition can be confirmed rather easily. 

\begin{proposition} \label{prop:stdecomp} 
For a digraph $G$, the binary relation $\st$ is a partial order over $\stcomp{G}$. 
\end{proposition} 

We call the partially ordered set $(\stcomp{G}, \st)$ the {\em strong component decomposition} of $G$ 
and denote this by $\stset{G}$. 
The next statement is easily confirmed. 

\begin{proposition} \label{prop:stcomp2dipath} 
Let $G$ be a digraph. 
For any $u,v\in V(G)$, 
there is a dipath from $u$ to $v$ 
if and only if $D_1 \st D_2$, where $D_1, D_2\in\stcomp{G}$ are the strong components 
with $u\in D_1$ and $v\in D_2$. 
\end{proposition} 

Proposition~\ref{prop:stcomp2dipath} means that 
the strong component decomposition characterizes how  dipaths exist in a digraph. 
The ``converse'' of the strong component decomposition also holds.  
That is,  
given a set $\mathcal{D}$ of strongly connected digraphs and a partial order $\preccurlyeq$ over $\mathcal{D}$, 
we can construct a new  digraph $G$ so that  
the strong component decomposition of $G$ is $(\mathcal{D}, \preccurlyeq)$.

\begin{proposition} \label{prop:order2stcomp} 
Let $\mathcal{D}$ be a set of strongly connected digraphs that are pairwise disjoint, 
and let $\preccurlyeq$ be a partial order over $\mathcal{D}$. 
Let $G$ be a digraph obtained by the following procedure: 
\begin{rmenum} 
\item If $D_1, D_2\in\mathcal{D}$ are nonrefinable, that is, 
$D_1\preccurlyeq C \preccurlyeq D_2$ implies $D_1 = C$ or $D_2 = C$ for every $C\in\mathcal{D}$, 
then add an arc whose tail and head are in $D_1$ and $D_2$, respectively. 
\item 
For arbitrary two digraphs $D_1, D_2\in\mathcal{D}$ with $D_1\preccurlyeq D_2$, 
arbitrarily add arcs whose tail and head are in $D_1$ and $D_2$, respectively;  this step can be skipped. 
\end{rmenum} 
Then, $\stset{G}$ is identical to $(\mathcal{D}, \preccurlyeq)$. 
\end{proposition} 


Additionally,  a constructive characterization for strongly connected digraphs is known.

\begin{theorem} \label{thm:ear}  
A digraph  $G$ is strongly connected  if and only if $G$ is a member of $\sdigraph$ that is defined as follows: 
\begin{rmenum} 
\item 
A digraph with only one vertex is a member of $\sdigraph$. 
\item 
Let $D \in \sdigraph$,  and $P$ be a (possibly closed) ditrail 
whose ends are vertices of $D$ and arcs are disjoint from the arcs of $D$. 
Then, $D+P$ is a member of $\sdigraph$. 
\end{rmenum}
\end{theorem}  


Proposition~\ref{prop:order2stcomp} and Theorem~\ref{thm:ear} mean 
that we can construct a digraph with a desired strong connectivity from scratch.

\section{Digraphic Bidirected Graphs} \label{sec:digraphic}

A bidirected digraph is {\em digraphic} if every edge is a $(+, -)$-edge.  
Let $G$ be a digraphic bidirected graph.  
There are two ways to regard $G$ as a digraph. 
For $\alpha \in \{+, -\}$,   
we call $G$ an $\alpha$-digraphic bidirected graph 
if we consider $G$ as a digraph 
by  regarding each $(\alpha, -\alpha)$-edge 
as an arc whose tail and head are  the ends with signs $\alpha$ and $-\alpha$, respectively.

The strong connectivity and strong components can be defined straightforwardly  for digraphic bidirected graphs. 
These concepts are uniquely determined regardless of the choice of  sign $\alpha$. 
That is, two vertices $u$ and $v$ of $G$ are {\em strongly connected} 
if $G$ has  $(-, +)$-ditrails from $u$ to $v$ and from $v$ to $u$. 
We say that $G$  is {\em strongly connected} if every two vertices are strongly connected. 
A {\em strongly connected component} or {\em strong component}  of $G$ 
is a maximal strongly connected subgraph. 
We also denote the set of strong components of $G$ by $\stcomp{G}$.

In contrast, for a digraphic bidirected graph, 
the strong component decomposition can be defined in two ways, 
depending on the choice of  the sign $\alpha$.   
Now, let $\alpha\in\{+, -\}$, and let $G$ be an $\alpha$-digraphic bidirected graph. 
The binary relation $\st_\alpha$ over $\stcomp{G}$ is defined  in the same way as $\st$. 
That is, for $C, D\in\stcomp{G}$,  
we let $C \st_\alpha D$ if $C \st D$ holds in the $\alpha$-digraphic bidirected graph $G$. 
Proposition~\ref{prop:stdecomp} obviously implies that $\st_\alpha$ is a partial order over $\stcomp{G}$. 
Accordingly,  we denote the poset  $(\stcomp{G}, \st_\alpha)$  by $\parstset{G}{\alpha}$  
and call this the strong component decomposition of the $\alpha$-digraphic bidirected graph. 
Analogues of Proposition~\ref{prop:order2stcomp} and Theorem~\ref{thm:ear} also hold for $\alpha$-digraphic bidirected graphs.

\section{Flowgraphs} \label{sec:flowgraph}

Let $G$ be a digraph, and let $r\in V(G)$. 
The digraph $G$ is called a {\em flowgraph} with {\em root} $r$ 
if, for every $x\in V(G)$,  there is a directed trail from $x$ to $r$. 

Under Proposition~\ref{prop:stcomp2dipath},  it is easily observed that flowgraphs can be characterized as follows.  
A digraph $G$ with $r\in V(G)$ is a flowgraph with root $r$  
if and only if  
 $\stset{G}$ has the maximum element $C\in\stcomp{G}$, and the vertex $r$ is contained in $C$.  

Let $\alpha\in\{+, -\}$. 
An $\alpha$-digraphic bidirected graph is an $\alpha$-flowgraph with root $r$ 
if it is a flowgraph with root $r$. 
The characterization of flowgraphs also applies to $\alpha$-flowgraphs. Hence, the next proposition holds. 

\begin{proposition} \label{prop:flowgraph2char} 
Let $\alpha \in \{+, -\}$. 
An $\alpha$-digraphic bidirected graph $G$ with $r\in V(G)$ is an $\alpha$-flowgraph with root $r$  
if and only if  
 $\parstset{G}{\alpha}$ has the maximum element $C\in\stcomp{G}$, and the vertex $r$ is contained in $C$.  
\end{proposition} 

Note that, under Proposition~\ref{prop:order2stcomp}, Theorem~\ref{thm:ear}, and Proposition~\ref{prop:flowgraph2char}, 
we can construct any $\alpha$-flowgraphs from scratch.

\section{Critical Graphs} \label{sec:critical}

Let $G$ be an (undirected) graph. 
Let $b: V(G) \rightarrow \nonnegz$. 
A set of edges $F\subseteq E(G)$ is a {\em $b$-factor} 
if, for each $v\in V(G)$, the number of edges from $F$ that are adjacent to $v$ is $b(v)$.  
A $b$-factor does not necessarily exist in the graph. 
For $x\in V(G)$, $b^x$ denotes a mapping $V(G) \rightarrow \nonnegz$ 
such that $b^x(v) = b(v)$ for each $v\in V(G)\setminus \{x\}$ 
and $b^x(x) = b(x) - 1$.  
We say that $G$ is {\em $b$-critical} or {\em critical} 
if, for each $x\in V(G)$,  there is a $b^x$-factor in $G$.

There is an easy one-to-one correspondence between a signed graph and a pair of a graph and a set of edges.  
Let $G$ be a graph, and let $F\subseteq E(G)$.  
The bidirected graph $G^F$ denotes the signed graph 
 with  $V(G^F) = V(G)$ and $E(G^F) = E(G)$ such that 
 $e\in E(G^F)$ is a $(-, -)$-edge for each $e\in F$, but is a $(+, +)$-edge for each $e\in E(G)\setminus F$. 

\begin{observation} \label{obs:critical2radial} 
Let $G$ be an (undirected) graph, and let $b: V(G) \rightarrow \nonnegz$.   
Let $r\in V(G)$, and let $F \subseteq E(G)$ be a $b^r$-factor of $G$. 
Then, $G$ is $b$-critical if and only if, 
for each $x\in V(G^F)$, 
there is a $(-, +)$-ditrail from $x$ to $r$ in $G^F$. 
\end{observation}

\section{Radials and Semiradials} \label{sec:radial} 

Under Observation~\ref{obs:critical2radial}, 
we define the concept of radials 
as a common generalization of flowgraphs and critical graphs. 
We also define semiradials, which is a relaxed concept of radials.

\begin{definition} 
Let $G$ be a bidirected graph,  let $r\in V(G)$, and let $\alpha\in\{+, -\}$. 
We call $G$ an {\em $\alpha$-radial} with {\rm root} $r$ 
if, for every $v\in V(G)$,  there is an $(\alpha, -\alpha)$-ditrail from $v$ to $r$.  
We call $G$ an {\em $\alpha$-semiradial} with {\rm root} $r$ 
if, for every $v\in V(G)$,  there is an $\alpha$-ditrail from $v$ to $r$.  
\end{definition}  

From the definition, any $\alpha$-radial  is an $\alpha$-semiradial.

\section{Absolute Semiradials} \label{sec:absolute} 

In this section, we define absolute semiradials 
and provide their constructive characterization in Theorem~\ref{thm:asr} 
that is similar to ear decomposition.

\begin{definition} 
Let $G$ be a bidirected graph, and let $r\in V(G)$. 
We call $G$ an {\em absolute semiradial} with root $r$ 
if $G$ is an $\alpha$-semiradial with root $r$ for each $\alpha\in\{+, -\}$. 
\end{definition}

\begin{definition} 
We define a set $\sqrset{r}$ of bidirected graphs with vertex $r$ as follows: 
\begin{rmenum} 
\item 
The graph that consists of a single vertex $r$ and no edge is a member of $\sqrset{r}$. 
\item 
Let $H\in\sqrset{r}$, and 
let $P$ be a diear relative to $H$. 
Then, $H + P$ is a member of $\sqrset{r}$. 
\end{rmenum} 
\end{definition}

The following theorem is the constructive characterization of absolute semiradials.

\begin{theorem} \label{thm:asr} 
Let $r$ be a vertex symbol. 
Then, $\sqrset{r}$ is the set of absolute semiradials with root $r$. 
\end{theorem}

In the following, we prove Theorem~\ref{thm:asr}. 
The next lemma proves a half of Theorem~\ref{thm:asr}. 

\begin{lemma} \label{lem:const2asr} 
If a bidirected graph $G$ is a member of $\sqrset{r}$, 
then $G$ is an absolute semiradial. 
\end{lemma} 

\begin{proof} 
We proceed by induction along the constructive definition of $\sqrset{r}$.  
For the base case, where $V(G)  = \{r\}$, the statement trivially holds. 
For proving the induction case, 
let $G \in \sqrset{r}$, and assume that $G$ is an absolute semiradial with root $r$. 
Let $P$ be a diear relative to $G$, and let $\hat{G} := G + P$. 
If $x\in V(\hat{G})$ is from $V(G)$, 
then there are clearly $+$- and $-$-ditrails from $x$ to $r$. 
Next, consider the case where  $x\in V(\hat{G})\cap V(P)$. 
For each $\alpha \in \{+, -\}$, 
it can  be  easily confirmed that $P$ contains an $\alpha$-ditrail $Q$ from $x$ to a vertex $y\in V(G)\cap V(P)$. 
As $G$ is an absolute semiradial, 
$G$ has an $-\se{y}{Q}$-ditrail $R$ from $y$ to $r$.  
Thus, $Q + R$ is an $\alpha$-ditrail from $x$ to $r$. 
This completes the proof.

\end{proof} 


In the following, we prove the remaining half of Theorem~\ref{thm:asr}. 
The next lemma is for proving Lemma~\ref{lem:asr2set}.

\begin{lemma} \label{lem:asr2const} 
If $G$ is an absolute semiradial with root $r$, 
then, for any subgraph $H$ of $G$ with $H \neq G$ that is an absolute semiradial with root $r$, 
there is a diear relative to $H$.  
Furthermore, if $V(H) \subsetneq V(G)$ holds, 
then there is a diear relative to $H$ that has a vertex in $V(G)\setminus V(H)$. 
\end{lemma} 

\begin{proof} 
If $V(H) = V(G)$,  then any edge from $E(G)\setminus E(H)$ forms a diear relative to $H$. 
Hence, in the following, consider the case where $V(H)\subsetneq V(G)$. 
As $G$ is obviously connected, there is an edge $e\in \parcut{G}{H}$. 
Let $x\in V(H)$ and $y\in V(G)\setminus V(H)$ be the ends of $e$, 
and let $\beta$ and $\gamma$ be the signs of $x$ and $y$ over $e$, respectively. 
Let $P$ be a $-\gamma$-ditrail in $G$ from $y$ to $r$. 
Trace $P$ from $y$, and let $z$ be the first encountered vertex in $V(H) \cup \{y\}$. 
If $z$ is equal to $y$, then $(x, e, y) + yPz + (z, e, x)$ forms a scoop diear relative to $H$ whose grip is $e$. 
Otherwise, $(x, e, y) + yPz$ is a simple diear relative to $H$.   
This completes the proof. 
\end{proof} 


The next lemma proves the remaining half of Theorem~\ref{thm:asr}. 

\begin{lemma} \label{lem:asr2set} 
If $G$ is an absolute semiradial with root $r\in V(G)$, 
then $G\in\sqrset{r}$ holds. 
\end{lemma} 

\begin{proof} 
Let $H$ be a maximal subgraph of $G$ that is an absolute semiradial with root $r$; 
such $H$ certainly exists, because $G[r]$ is trivially an absolute semiradial. 
Suppose $H \neq G$. 
Lemma~\ref{lem:asr2const} implies that there is a diear $P$ relative to $H$. 
According to Lemma~\ref{lem:const2asr}, $H + P$ is a subgraph of $G$ that is an absolute semiradial with root $r$. 
This contradicts the maximality of $H$. Hence, $H = G$. 
The proof is complete.
\end{proof} 


From Lemmas~\ref{lem:const2asr} and \ref{lem:asr2set}, Theorem~\ref{thm:asr} is now proved.

\section{Strong Radials} \label{sec:strong} 

In this section, we define the concept of {\em strong} radials, 
which is a special class of absolute semiradials, 
and provide their constructive characterization, 
which is a special case of Theorem~\ref{thm:asr}.

\begin{definition} 
Let $\alpha\in\{+, -\}$. 
An $\alpha$-radial $G$ with root $r\in V(G)$ is said to be {\em strong} 
if,  for every $v\in V(G)$,  there is a $(-\alpha, -\alpha)$-ditrail from $v$ to $r$. 
\end{definition} 

Note that, from the definition, any strong radial is an absolute semiradial. 

\begin{definition} 
Let $r$ be a vertex symbol, and let $\alpha \in \{+, -\}$. 
We define a set $\acset{r}{\alpha}$ of bidirected graphs that have vertex $r$ as follows: 
\begin{rmenum} 
\item \label{item:base} 
A $(-\alpha, -\alpha)$-simple diear  relative to $r$  is an element of $\acset{r}{\alpha}$. 
\item \label{item:inductive}
If $G\in \acset{r}{\alpha}$ holds and $P$ is a diear   relative to $G$, 
then $G + P \in \acset{r}{\alpha}$ holds. 
\end{rmenum}
\end{definition}

The next lemma characterizes the relationship between strong radials and absolute semiradials.

\begin{lemma} \label{lem:str2absr} 
Let $\alpha\in\{+,-\}$. 
A bidirected graph $G$ with a vertex $r\in V(G)$ is a strong $\alpha$-radial with root $r$ 
if and only if 
$G$ is an absolute semiradial with root $r$ 
that has  a $(-\alpha, -\alpha)$-closed ditrail over $r$. 
\end{lemma} 

\begin{proof} 
The sufficiency  obviously holds. 
For proving the necessity, 
let $x\in V(G)$,  let $P$ be an $\alpha$-ditrail of $G$ from $x$ to $r$, 
and let $C$ be a $(-\alpha, -\alpha)$-closed ditrail over $r$.  
Trace $P$ from $x$, and let $y$ be the first encountered vertex in $C$. 
Then, either $xPy + yCr$ or $xPy + y\reverse{C}r$ is an $(\alpha, -\alpha)$-ditrail from $x$ to $r$. 
Hence, $G$ is a strong $\alpha$-radial with root $r$. 
This completes the proof. 
\end{proof} 


Lemma~\ref{lem:str2absr} derives Theorem~\ref{thm:str} as follows. 

\begin{theorem} \label{thm:str} 
Let $r$ be a vertex symbol. 
Then, $\acset{r}{\alpha}$ is the set of strong $\alpha$-radials with root $r$. 
\end{theorem} 
\begin{proof} 
Let $G\in \acset{r}{\alpha}$. 
Under Theorem~\ref{thm:asr},  
$G$ is an absolute semiradial with root $r$ that has a $(-\alpha, -\alpha)$-closed ditrail over $r$. 
Thus, Lemma~\ref{lem:str2absr} implies that $G$ is a strong $\alpha$-radial with root $r$. 

Next, let $G$ be a strong $\alpha$-radial with root $r$. 
According to Lemma~\ref{lem:str2absr}, 
$G$ is an absolute semiradial with root $r$ and has a $(-\alpha, -\alpha)$-closed ditrail $C$ over $r$.  
Obviously,  $C$ itself is an absolute semiradial with root $r$.  
Hence, Lemma~\ref{lem:asr2const} implies that 
$G$ can be constructed as a member of $\sqrset{r}$ in which $C$ is the initial diear relative to $r$. 
That is, $G$ is a member of $\acset{r}{\alpha}$. 
\end{proof} 

\section{Almost Strong Radials} \label{sec:astr} 
\subsection{Characterization of Almost Strong Radials} \label{sec:astr:overview} 

In this section, we define another special class of absolute semiradials, {\em almost strong radials}, 
and provide their constructive characterization. 
We show that almost strong radials are constructed using strong radials.

\begin{definition} 
Let $\alpha \in \{+, -\}$. 
An $\alpha$-radial $G$ with root $r$ is said to be {\em almost strong} 
if, for every $v\in V(G)\setminus \{r\}$, 
there is a $(-\alpha, -\alpha)$-ditrail from $v$ to $r$, 
however there is no closed $(-\alpha, -\alpha)$-ditrail over $r$. 
\end{definition} 

\begin{definition} 
Let $\alpha\in\{+, -\}$. 
Define a set $\nacset{r}{\alpha}$ of bidirected graphs with a vertex $r$ as follows: 
\begin{rmenum} 
\item \label{item:const:base} Let $\beta \in \{+, -\}$,  let $r'$ be a vertex symbol distinct from $r$, 
let $G\in \acset{r'}{\beta}$ be a bidirected graph with $r\not\in V(G)$, 
and let $rr'$ be an edge for which the signs of $r$ and $r'$ are $-\alpha$ and $\beta$, respectively. 
Then, $G + rr'$ is a member of $\nacset{r}{\alpha}$. 
\item \label{item:const:edge}
Let $G\in \nacset{r}{\alpha}$, 
let $v\in V(G)$,  
and let $rv$ be an edge in which the sign of $r$ is  $\alpha$. 
Then, $G + rv$ is a member of $\nacset{r}{\alpha}$. 
\item \label{item:const:vertex} 
Let $G_1, G_2\in \nacset{r}{\alpha}$ be bidirected graphs with $V(G_1)\cap V(G_2) = \{r\}$. 
Then, $G_1 + G_2$ is a member of $\nacset{r}{\alpha}$.  
\end{rmenum} 
\end{definition} 

The following theorem is our constructive characterization of almost strong radials. 

\begin{theorem} \label{thm:asr2char} 
Let $r$ be a vertex symbol. 
Let $\alpha \in \{+, -\}$.  
Then, $\nacset{r}{\alpha}$ is the set of almost strong $\alpha$-radials with root $r$. 
\end{theorem}


Theorem~\ref{thm:asr2char} can be immediately obtained  from the following two lemmas. 

\begin{lemma} \label{lem:critical2const} 
Let $r$ be a vertex symbol. 
Let $\alpha\in\{+, -\}$. 
Any almost strong $\alpha$-radial $G$ is an element of $\nacset{r}{\alpha}$. 
\end{lemma} 


\begin{lemma} \label{lem:const2critical} 
Let $r$ be a vertex symbol. 
Let $\alpha\in\{+, -\}$. 
Then, any element of $\nacset{r}{\alpha}$ is an almost strong $\alpha$-radial with root $r$. 
\end{lemma} 


In Sections~\ref{sec:astr:decomp} and \ref{sec:astr:const},  we prove Lemmas~\ref{lem:critical2const} and \ref{lem:const2critical}, respectively, 
and thus complete the proof of Theorem~\ref{thm:asr2char}.

\subsection{Decomposition of Almost Strong Radials} \label{sec:astr:decomp}

This section is devoted to proving Lemma~\ref{lem:critical2const}. 

In the following, we provide and prove Lemmas~\ref{lem:rootsimple} to \ref{lem:soleneck} 
and thus prove Lemma~\ref{lem:critical2const}. 

The next lemma is used for deriving Lemmas~\ref{lem:asr2edgedel} and \ref{lem:asr2baloon}. 

\begin{lemma} \label{lem:rootsimple}
Let $\alpha \in \{+, -\}$. 
Let $G$ be an almost strong $\alpha$-radial with root $r\in V(G)$. 
For any $x\in V(G)\setminus \{r\}$, $G$ has $(\alpha, -\alpha)$- and $(-\alpha, -\alpha)$-ditrails from $x$ to $r$  
in each of which the vertex $r$ is contained only once. 
\end{lemma} 

\begin{proof} 
Let $\beta\in\{+, -\}$ and  $x\in V(G)\setminus \{r\}$. 
Let $P$ be a $(\beta, -\alpha)$-ditrail from $x$ to $r$. 
Trace $P$ from $x$, and let $s$ be the first encountered vertex that is equal to $r$. 
If the sign of $s$ over $xPs$ is $\alpha$, then $sPr$ is a $(-\alpha, -\alpha)$-closed ditrail over $r$, which is a contradiction. 
Hence, the sign of $s$ over $xPs$ is $-\alpha$, 
and accordingly,  $xPs$ is a desired ditrail. 
\end{proof} 


Lemma~\ref{lem:rootsimple} easily implies the next two lemmas. 
These two lemmas are directly used  for proving Lemma~\ref{lem:critical2const}. 

\begin{lemma} \label{lem:asr2edgedel} 
Let $\alpha \in \{+, -\}$. 
Let $G$ be an almost strong $\alpha$-radial with root $r$. 
Let $F\subseteq \parcut{G}{r}$ be a set of edges in which the sign of $r$ is $\alpha$. 
Then, $G-F$ is an almost strong $\alpha$-radial with root $r$. 
\end{lemma}

\begin{lemma} \label{lem:asr2blockdecomp} 
Let $\alpha \in \{+, -\}$. 
Let $G$ be an almost strong $\alpha$-radial with root $r\in V(G)$. 
Then, each block over $r$ is an almost strong $\alpha$-radial with root $r$. 
\end{lemma} 


The next lemma is provided for proving Lemma~\ref{lem:soleneck}. 

\begin{lemma} \label{lem:asr2baloon} 
Let $\alpha \in \{+, -\}$. 
Let $G$ be an almost strong $\alpha$-radial with root $r$. 
If $V(G)\setminus \{r\} \neq \emptyset$, then 
$\parcut{G}{r}$ contains an edge $e$ in which the sign of $r$  is $-\alpha$.  
Furthermore, for such $e$, there is a $-\alpha$-scoop diear relative to $r$ whose grip is $e$. 
This diear does not have vertex terms that denote $r$ except for the first and last ones. 
\end{lemma} 

\begin{proof} 
The first statement is obvious from the assumption on $G$. 
Let $e\in\parcut{G}{r}$ be an edge in which the sign of $r$ is $-\alpha$,  let $x$ be the end of $e$ other than $r$, 
and let $\beta \in \{+, -\}$ be the sign of $x$ over $e$. 
According to the assumption on $G$, there is a $(-\beta, -\alpha)$-ditrail $P$ from $x$ to $r$.  
\begin{pclaim} \label{claim:asr2baloon:rootsimple}
$P$ does not contain any $-\beta$-ditrail from $x$ to $r$ that does not contain $e$. 
\end{pclaim} 
\begin{proof} 
Suppose that $P'$ is a $-\beta$-ditrail from $x$ to $r$ without $e$. 
If $P'$ is a $(-\beta, -\alpha)$-ditrail, then 
$(r, e, x) + P$ is a $(-\alpha, -\alpha)$-closed ditrail over $r$, 
which is a contradiction. 
If $P'$ is a $(-\beta, \alpha)$-ditrail, 
then $P$ contains a $(-\alpha, -\alpha)$-closed ditrail over $r$, which is again a contradiction. 
Hence, the claim is proved. 
\end{proof} 

Claim~\ref{claim:asr2baloon:rootsimple} implies that 
$P$ contains a subtrail  $C + (x, e, r)$, 
where $C$ is a $(-\beta, -\beta)$-closed ditrail over $x$ that does not contain $r$.  
Therefore, $(r, e, x) + C + (x, e, r)$ forms a $-\alpha$-scoop diear that meets the claim.

\end{proof} 


Lemma~\ref{lem:asr2baloon}, together with Theorem~\ref{thm:str},  
derives the next lemma. This lemma is directly used for proving Lemma~\ref{lem:critical2const}. 

\begin{lemma} \label{lem:soleneck}
Let $\alpha \in \{+, -\}$. 
Let $G$ be an almost strong $\alpha$-radial with root $r$. 
Then, for each block $C$ over $r$,  
$\parcut{C}{r}$ contains exactly one edge $e_C$ such that the sign of $r$ over $e_C$ is $-\alpha$. 
Furthermore, 
$C - r$ is a strong $\beta_C$-radial with root $x_C$, 
where $x_C$ is the end of $e$ other than $r$ and $\beta_C$ is the sign of $x_C$ over $e_C$. 
\end{lemma} 

\begin{proof}
Let $F$ be the set of edges from $\parcut{G}{r}$ in which  the sign of $r$ is $-\alpha$. 
For each $i\in F$, let $x_i$ be the end of $i$ other than $r$, let $\beta_i$ be the sign of $x_i$ over $i$, 
and $K_i$ be the maximal subgraph of $G$ that is a strong $\beta_i$-radial with root $x_i$.  
Lemma~\ref{lem:asr2baloon} implies that there is a $(-\beta_i, -\beta_i)$-closed ditrail over $x_i$; 
 therefore, Theorem~\ref{thm:str}  ensures that $K_i$ is nonempty. 

\begin{pclaim} \label{claim:soleneck:nonpara} 
If $e$ and $f$ are distinct edges from $F$, then $x_e$ and $x_f$ are distinct. 
\end{pclaim} 
\begin{proof} 
Suppose that $x_e = x_f$, that is, $e$ and $f$ are parallel edges. 
If $\beta_e \neq \beta_f$, 
then $(r, e, x_e) + (x_f, f, r)$ is a $(-\alpha, -\alpha)$-closed ditrail over $r$; 
this is a contradiction. 
Hence, $\beta_e = \beta_f$. 
Lemma~\ref{lem:asr2baloon} implies that 
there is a $(-\beta_e, -\beta_e)$-closed ditrail $C$ over $x_e$ 
and that $C$ does not contain $r$. Note that this implies $C$ does not contain $f$. 
Hence, 
$(r, e, x_e) + C + (x_f, f, r)$ is a $(-\alpha, -\alpha)$-closed ditrail over $r$, 
which is again a contradiction. 
Thus, the claim is proved. 
\end{proof} 

Hence, in the following, we assume Claim~\ref{claim:soleneck:nonpara}. 

\begin{pclaim} \label{claim:soleneck:disjoint} 
If $e$ and $f$ are distinct edges from $F$, then $K_e$ and $K_f$ are disjoint. 
\end{pclaim} 
\begin{proof} 
Suppose that the claim fails, and let $u \in V(K_e)\cap V(K_f)$. 
As $K_f$ is a strong $\beta_f$-radial, there is an $(\alpha, -\beta_f)$-ditrail $P$ of $K_f$ from $u$ to $x_f$. 
Trace $\reverse{P}$ from $x_f$, and 
let $v$ be the first encountered vertex in $V(K_e)$. 
Let $\gamma \in \{+, -\}$ be the sign of $v$ over $x_f\reverse{P}v$.  
Because $K_e$ is a strong $\beta_e$-radial, there is a $(-\gamma, -\beta_e)$-ditrail $Q$ of $K_e$ from $v$ to $x_e$. 
Thus, $(r, e, x_e) + \reverse{Q} + vPx_f + (x_f, f, r)$ is a $(-\alpha, -\alpha)$-closed ditrail over $r$, 
which contradicts the assumption that $G$ is an almost strong $\alpha$-radial. 
\end{proof}  

Hence, in the following, we assume Claim~\ref{claim:soleneck:disjoint}. 

\begin{pclaim} \label{claim:soleneck:nobridge} 
If $e$ and $f$ are distinct edges from $F$, 
 $G$ has no ditrail  whose ends are individually in $K_e$ and $K_f$ 
 and whose edges are disjoint from $E(K_e)\cup E(K_f) \cup \{e\} \cup \{f\}$. 
\end{pclaim} 
\begin{proof} 
Suppose, to the contrary, that $G$ has such ditrail $R$ from $v_e \in V(K_e)$ to $v_f \in V(K_f)$.  
Let $\gamma_e$ and $\gamma_f$  be the signs of $v_e$ and $v_f$ over $R$, respectively. 
Because $K_i$ is a strong $\beta_i$-radial for each each $i \in \{e, f\}$, 
there is a $(-\gamma_i, -\beta_i)$-ditrail $P_i$ of $K_i$ from $v_i$ to $x_i$. 
Then, $(r, e, x_e) + \reverse{P_e} +  R + P_f + (x_f, f, r)$ is a $(-\alpha, -\alpha)$-closed ditrail over $r$, 
which contradicts the assumption that $G$ is an almost strong $\alpha$-radial. 
\end{proof}  

Note that Claim~\ref{claim:soleneck:nobridge} particularly implies that 
there is no edge between $K_e$ and $K_f$. 

Let $\bigcup_{i\in F} V(K_i) =: U$. 
It remains to prove  $U = V(G)\setminus \{r\}$. 


\begin{pclaim} \label{claim:soleneck:all}
For each $x\in V(G)\setminus \{r\}$, there exists $e\in F$ with $x\in V(K_e)$. 
\end{pclaim} 
\begin{proof} 
Suppose, to the contrary, that the set $V(G)\setminus \{r\} \setminus U$ is not empty. 
As $G$ is a radial, there is a ditrail from any vertex in this set to the vertex $r$. 
Therefore, 
there exists $v\in U$ and $u\in V(G)\setminus \{r\} \setminus U$ that are adjacent by an edge $uv$. 
Let $e\in F$ be the edge such that  $v\in V(K_e)$ holds. 
Let $\beta \in \{+, -\}$ be the sign of $u$ over the edge $uv$. 
There is a $(-\beta, -\alpha)$-ditrail $P$ from $u$ to $r$.  
Trace $P$ from $u$, and let $x$ be the first encountered vertex in $U \cup \{r\}$; 
from the definition of $U$, $x\neq r$ holds, 
and therefore, $uPx$ does not contain any edges from $F$. 
Let $f \in F$ be the edge such that $x \in V(K_f)$ holds.  
Then, $(v, uv, u) + uPx$ is a ditrail that joins $K_e$ and $K_f$ 
 and is disjoint from $E(K_e)\cup E(K_f)\cup \{e\} \cup \{f\}$. 
Claim~\ref{claim:soleneck:nobridge} implies  $e = f$. 
Hence, $(v, uv, u) + uPx$ is a diear relative to $K_e$. 
Theorem~\ref{thm:str} then implies that $K_e + ( (v, uv, u) + uPx)$ is also a strong $\beta_e$-radial with root $x_e$, 
which contradicts the maximality of $K_e$. 
The claim is now proved. 

\end{proof} 

Combining Claims~\ref{claim:soleneck:disjoint}, \ref{claim:soleneck:nobridge}, and \ref{claim:soleneck:all},  
it is now derived that $\{K_i: i\in F\}$ coincides with the set of connected components of $G-r$. 
Thus, the remaining statement of this lemma is also proved. 
\end{proof} 


From Lemmas~\ref{lem:asr2edgedel}, \ref{lem:asr2blockdecomp}, and \ref{lem:soleneck}, 
we can now prove Lemma~\ref{lem:critical2const}.

\begin{proof}[Proof of Lemma~\ref{lem:critical2const}] 
First, consider the case where $G$ has only one block over $r$ and
  the sign of $r$ is $-\alpha$ over every edge in $\parcut{G}{r}$. 
According to Lemma~\ref{lem:soleneck}, 
$G$ is then a bidirected graph obtained by the construction~\ref{item:const:base}. 

For proving other cases, we proceed by the induction on $|V(G)| + |E(G)|$ where the above case serves as the base case. 
Under the induction hypothesis, Lemmas~\ref{lem:asr2edgedel} and \ref{lem:asr2blockdecomp} prove the remaining cases, 
namely, where $\parcut{G}{r}$ contains edges in which the sign of $r$ is $\alpha$, 
and where $G$ has multiple blocks over $r$. 
This prove the lemma. 

\end{proof} 


\subsection{Construction of Almost Strong Radials} \label{sec:astr:const} 

In this section, we prove Lemma~\ref{lem:const2critical}, that is, the remaining half of Theorem~\ref{thm:asr2char}.  
We provide and prove Lemmas~\ref{lem:asr2edgeadd} and \ref{lem:asr2blockadd} in the following. 
We then use these lemmas to prove Lemma~\ref{lem:const2critical}.

\begin{lemma} \label{lem:asr2edgeadd}
Let $\alpha \in \{+, -\}$, and let $G$ be an almost strong $\alpha$-radial with root $r$. 
Let $x\in V(G)$. 
Then, the  bidirected graph obtained by adding to $G$ an edge $e$ that joins $x$ and $r$  
 is an almost strong $\alpha$-radial with root $r$ if the sign of $e$ over $r$ is $\alpha$. 
\end{lemma} 

\begin{proof} 
Let $\hat{G} := G + e$. 
It is obvious that the lemma is proved if there is no $(-\alpha, -\alpha)$-closed ditrail over $r$. 
Suppose, to the contrary, that $\hat{G}$ has a $(-\alpha, -\alpha)$-closed ditrail $P$ over $r$. 
Then, $P$ contains the edge $e$. 
Assume, without loss of generality, that $P$ contains the sequence $(x, e, r)$. 
Then, the subtrail of $P$ that follows $(x, e, r)$ is an $(-\alpha, -\alpha)$-closed ditrail over $r$ without the edge $e$. This contradicts the assumption on $G$. 
The lemma is proved. 
\end{proof} 


\begin{lemma} \label{lem:asr2blockadd} 
Let $\alpha \in \{+, -\}$, and let $G_1$ and $G_2$ be almost strong $\alpha$-radial with root $r$ 
such that $V(G_1)\cap V(G_2) = \{ r\}$. 
Then, $G_1 + G_2$ is an almost strong $\alpha$-radial with root $r$. 
\end{lemma} 

\begin{proof} 
It suffices to prove that $G_1 + G_2$ does not have any $(-\alpha, -\alpha)$-closed ditrails over $r$. 
Suppose, to the contrary, that $G_1 + G_2$ has a $(-\alpha, -\alpha)$-closed ditrail $P$ over $r$. 
Then, $P$ is partitioned into closed ditrails over $r$ each  possessed by $G_1$ or $G_2$. 
A simple counting argument derives that one of them is $(-\alpha, -\alpha)$, which contradicts the assumption on $G_1$ or $G_2$. 
This proves the lemma. 
\end{proof} 


We can now prove Lemma~\ref{lem:const2critical} from Lemmas~\ref{lem:asr2edgeadd} and \ref{lem:asr2blockadd}. 

\begin{proof}[Proof of Lemma~\ref{lem:const2critical}] 
It is easily confirmed that a graph obtained by the construction \ref{item:const:base} is an almost strong $\alpha$-radial. 
We complete the proof of this lemma by the induction on the number of vertices and edges 
where the graphs obtained by the construction \ref{item:const:base} serve as the base case. 
Under the induction hypothesis, Lemmas~\ref{lem:asr2edgeadd} and \ref{lem:asr2blockadd} imply that  
the graphs obtained by the constructions \ref{item:const:edge} and \ref{item:const:vertex} are also 
almost strong $\alpha$-radials. 


\end{proof} 


This completes the proof of Theorem~\ref{thm:asr2char}.

\section{Linear Semiradials and Sublinear Radials} \label{sec:linear} 

\subsection{Definitions and Preliminaries} \label{sec:linear:pre} 

In this section, we introduce the classes of linear and sublinear semiradials. 
More precisely, 
we are interested in linear semiradials and sublinear radials, 
which are opposite classes of absolute semiradials and strong radials, respectively. 
We provide  characterizations of these two classes  
in Sections~\ref{sec:linear:lsr} and \ref{sec:linear:slr}. 
Here, in Section~\ref{sec:linear:pre}, 
we provide lemmas that are commonly used in Sections~\ref{sec:linear:lsr} and \ref{sec:linear:slr}.

\begin{definition} 
Let $\alpha\in\{+, -\}$. 
Let $G$ be an $\alpha$-semiradial with root $r$.  
We say that $G$ is {\em linear} if $G$ has no loop edge over $r$, and $G$ has no $-\alpha$-ditrails from $x$ to $r$
 for any $x\in V(G)$ 
except for the trivial $(-\alpha, \alpha)$-ditrail from $r$ to $r$ with no edge.  
We say that $G$ is {\em sublinear} if if $G$ has  no $(-\alpha, -\alpha)$-ditrails from $x$ to $r$ for any $x\in V(G)$. 
\end{definition} 

Note that if a semiradial $G$ is linear, then $G$ is sublinear; however, the converse does not hold. 
By definition, a linear $\alpha$-semiradial does not have any loops over $r$. 
A sublinear $\alpha$-semiradial cannot have a $(-\alpha, -\alpha)$-loop  over $r$. 
but may possess other kinds of loops  over $r$.

The next lemma is provided to be typically used for 
stating the structure of ditrails in linear semiradials and sublinear radials. 

\begin{lemma} \label{lem:ssr2ditrail} 
Let $\alpha\in\{+, -\}$. 
Let $G$ be a bidirected graph with a vertex $r\in V(G)$ such that 
no $(-\alpha, -\alpha)$-ditrail from $x$ to $r$ exists for any $x\in V(G)$. 
Let $x\in V(G)$, and let $P$ be an $(\alpha, -\alpha)$-ditrail from $x$ to $r$.  
Then, for any vertex term $w$ of $P$ except for the last one, 
$xPw$ is an $(\alpha, -\alpha)$-ditrail from $x$ to $w$,  
and $wPr$ is an $(\alpha, -\alpha)$-ditrail from $w$ to $r$. 
Accordingly, every edge in $P$ is an $(\alpha, -\alpha)$-edge. 
\end{lemma} 

\begin{proof} 
If $xPw$ is an $(\alpha, \alpha)$-ditrail, 
then $wPr$ is an $(-\alpha, -\alpha)$-ditrail from $w$ to $r$, 
which is a contradiction.  
Accordingly, the claims follow. 
\end{proof} 


Furthermore, the next lemma regarding ditrails 
is provided to be typically used for linear semiradials. 

\begin{lemma} \label{lem:lsr2ditrail} 
Let $\alpha\in\{+, -\}$. 
Let $G$ be a bidirected graph with a vertex $r\in V(G)$ such that 
$G$ has no loop over $r$, and  
no $-\alpha$-ditrail from $x$ to $r$ exists for any $x\in V(G)$ except for the trivial ditrail $(r)$.    
Let $x\in V(G)$, and let $P$ be an $(\alpha, \alpha)$-ditrail from $x$ to $r$.  
Then, for any vertex term $w$ of $P$ except for the last one, 
$xPw$ is an $(\alpha, -\alpha)$-ditrail from $x$ to $w$,  
and $wPr$ is an $(\alpha, \alpha)$-ditrail from $w$ to $r$. 
Accordingly, 
every edge in $P$ except the last one is  an $(\alpha, -\alpha)$-edge, 
whereas the last edge is an $(\alpha, \alpha)$-edge. 
\end{lemma} 

\begin{proof} 
The first statement obviously holds; 
for, otherwise, $wPr$ would be an $(-\alpha, \alpha)$-ditrail from $w$ to $r$.  
Accordingly, the claims follow.   
\end{proof} 


The next lemma is derived from Lemmas~\ref{lem:ssr2ditrail} and \ref{lem:lsr2ditrail}.

\begin{lemma}\label{lem:noposi} 
Let $\alpha\in\{+, -\}$. 
Let $G$ be a linear $\alpha$-semiradial or sublinear $\alpha$-radial with root $r\in V(G)$. 
Then, $G$ does not have any $(-\alpha, -\alpha)$-edges. 
\end{lemma} 

\begin{proof} 
Suppose, to the contrary, that $G$ has a $(-\alpha, -\alpha)$-edge $e$; let $u,v \in V(G)$ be the ends of $e$.  
Let $P$ be an $(\alpha, \beta)$-ditrail from $v$ to $r$; 
if $G$ is a linear $\alpha$-semiradial, then let $\beta\in\{+, -\}$; 
if $G$ is a sublinear $\alpha$-radial, then let $\beta = -\alpha$. 
Lemmas~\ref{lem:ssr2ditrail} and \ref{lem:lsr2ditrail} imply 
that  $P$ does not contain the edge $e$.  
Therefore, $(u, e, v) + P$ is an $(-\alpha, \beta)$-ditrail from $u$ to $r$, which is a contradiction. 
This completes the proof.

\end{proof} 


\subsection{Characterization of Linear Semiradials} \label{sec:linear:lsr} 

In this section, we prove Theorem~\ref{thm:lsr}, 
namely, a constructive characterization of linear semiradials. 
Here, the structure of a linear semiradial is stated 
using a digraphic bidirected graph and its strong component decomposition. 
It is revealed that 
linear semiradials are digraphic bidirected graphs 
with some homogeneous signed edges that do not very much affect the structure of ditrails. 

Lemmas~\ref{lem:lqr2decomp} and \ref{lem:lsr2const} in the following are provided for proving 
the sufficiency and necessity of Theorem~\ref{thm:lsr}, respectively. 

\begin{lemma} \label{lem:lqr2decomp} 
Let $\alpha\in\{+, -\}$. 
Let $G$ be a linear $\alpha$-semiradial with root $r\in V(G)$. 
Let $F$ be the set of $(\alpha, \alpha)$-edges in $G$. 
Then, $G-F$ is a digraphic digraph. 
Additionally, in the strong component decomposition of the $\alpha$-digraphic bidirected graph $G-F$, 
\begin{rmenum} 
\item \label{item:lqr2decomp:maxcomp} $(G-F)[r]$ is a strong component that is maximal, and 
\item \label{item:lqr2decomp:neigh} 
 each maximal strong component $C$ that is distinct from $(G-F)[r]$  
  has a vertex that is adjacent to $r$ with an $(\alpha, \alpha)$-edge in $G$. 
\end{rmenum} 
\end{lemma}

\begin{proof} 
The first statement  obviously follows from Lemma~\ref{lem:noposi}. 
Because $G$ has no $(-\alpha, \alpha)$-ditrail from any $x\in V(G)\setminus \{r\}$ to $r$, neither does $G-F$. 
This is equivalent to $(G-F)[r]$ being a maximal strong component of $G-F$; 
the statement \ref{item:lqr2decomp:maxcomp} is proved. 

For proving \ref{item:lqr2decomp:neigh}, 
let $C$ be a maximal strong component of $G-F$ that is distinct from $(G-F)[r]$,  
and let $x\in V(C)$. 
Let $P$ be an $\alpha$-ditrail of $G$ from $x$ to $r$.  
If $P$ is an $(\alpha, -\alpha)$-ditrail, 
then Lemma~\ref{lem:ssr2ditrail} implies that every edge of $P$ is $(\alpha, -\alpha)$, 
and accordingly, $P$ is also a ditrail of $G-F$; 
this contradicts the definitions of $C$ and $x$. 
Hence, $P$ is an $(\alpha, \alpha)$-ditrail.  
Let $e$ be the last edge of $P$, and let $z$ be the end of $e$ other than $r$. 
Lemma~\ref{lem:lsr2ditrail} implies that $xPz$ is an $(\alpha, -\alpha)$-ditrail from $x$ to $z$ 
and that $e$ is an $(\alpha, \alpha)$-edge. 
Because $x$ is chosen from the maximal strong component $C$,  
this implies  $z\in V(C)$, and accordingly, \ref{item:lqr2decomp:neigh} is proved.

\end{proof} 

\begin{lemma} \label{lem:lsr2const} 
Let $D$ be an $\alpha$-digraphic bidirected graph with $r\in V(D)$ 
in  which $D[r]$ is a maximal strong component.  
A bidirected graph $G$ constructed from $D$ as follows is a linear $\alpha$-semiradial with root $r$. 
\begin{rmenum} 
\item For each maximal strong component $C$ with $r\not\in V(C)$, 
arbitrarily choose a vertex $x\in V(C)$, and join $x$ and $r$ with an $(\alpha, \alpha)$-edge. 
\item 
Then, arbitrarily add $(\alpha, \alpha)$-edges except for loops over $r$; this operation can be skipped. 
\end{rmenum} 
\end{lemma} 

\begin{proof} 
It is easily confirmed that $G$ is an $\alpha$-semiradial with root $r$. 
For proving that $G$ is linear, 
suppose that $G$ has a nontrivial $(-\alpha, \beta)$-ditrail $P$ from $x\in V(G)$ to $r$, where $\beta$ is either $+$ or $-$. 
Because $G$ does not contain $(-\alpha, -\alpha)$-edges, 
$\beta = \alpha$; furthermore, every edge of $P$ is a $(-\alpha, \alpha)$-edge. 
Therefore, $P$ is also a ditrail of $D$.  
This contradicts that $D[r]$ is a maximal strong component of $D$. 
Hence, $G$ is linear. 
\end{proof} 


Combining Lemmas~\ref{lem:lqr2decomp} and \ref{lem:lsr2const}, 
Theorem~\ref{thm:lsr} is now proved.

\begin{theorem} \label{thm:lsr} 
A bidirected graph is a linear $\alpha$-semiradial with root $r$ 
if and only if it is constructed as follows: 
Let $D$ be an $\alpha$-digraphic bidirected graph with $r\not\in V(D)$.  
\begin{rmenum} 
\item For each maximal strong component $C$ of $D$, 
arbitrarily choose a vertex $u\in V(C)$, 
and join $u$ and $r$ with an $(\alpha, -\alpha)$- or $(\alpha, \alpha)$-edge 
in which the sign of $u$ is $\alpha$.  
\item 

Then, arbitrarily add $(\alpha, \alpha)$-edges except for loops over $r$; this operation can be skipped. 
\end{rmenum} 
\end{theorem} 


Under Theorem~\ref{thm:lsr} and the statements from Sections~\ref{sec:digraph} and \ref{sec:digraphic}, 
 the structure of linear semiradials are now understood from first principles.

\subsection{Characterization of Sublinear Radials}  \label{sec:linear:slr} 

In this section, 
we provide and prove Theorem~\ref{thm:slr2char}, namely, the constructive characterization of sublinear radials. 
In the following, we provide Lemmas~\ref{lem:nonposi2slsr}, \ref{lem:slr2decomp}, and \ref{lem:const2lsr} 
and use these to prove Theorem~\ref{thm:slr2char}. 

The next lemma proves the sufficiency of Theorem~\ref{thm:slr2char}. 

\begin{lemma} \label{lem:slr2decomp} 
Let $\alpha\in\{+, -\}$, and let $G$ be a  sublinear $\alpha$-radial with root $r\in V(G)$. 
Let $F$ be the set of $(\alpha, \alpha)$-edges in $G$. 
Then, $G-F$ is a sublinear $\alpha$-radial with root $r\in V(G)$ that is an $\alpha$-flowgraph. 
\end{lemma} 

\begin{proof}  
Lemma~\ref{lem:ssr2ditrail} proves that $G-F$ is also a sublinear $\alpha$-radial. 
Lemma~\ref{lem:noposi} implies that  $G-F$ is a digraphic digraph. 
Thus, the claim follows. 
\end{proof} 


The next lemma is an easy observation to be used for proving Lemma~\ref{lem:const2lsr}. 

\begin{lemma} \label{lem:nonposi2slsr} 
Let $\alpha\in\{+, -\}$. 
If an $\alpha$-semiradial does not have any $(-\alpha, -\alpha)$-edges, then it is sublinear. 
\end{lemma} 

\begin{proof} 
The statement obviously holds, because any $(-\alpha, -\alpha)$-ditrail has to contain an $(-\alpha, -\alpha)$-edge. 
\end{proof} 


Note that Lemmas~\ref{lem:noposi} and \ref{lem:nonposi2slsr} imply the next characterization of sublinear radials. 
The next proposition is provided for better understanding of sublinear semiradials. 

\begin{proposition} \label{prop:slsr2char} 
Let $\alpha\in\{+, -\}$. 
An $\alpha$-semiradial is sublinear if and only if there is no $(\alpha, \alpha)$-edge. 
\end{proposition}  


The next lemma is derived from Lemma~\ref{lem:nonposi2slsr}  and proves the necessity of Theorem~\ref{thm:slr2char}. 

\begin{lemma} \label{lem:const2lsr} 
Let $\alpha\in\{+, -\}$, and let $G$ be an $\alpha$-flowgraph with root $r$. 
Then,  
 a graph obtained from $G$ by arbitrarily adding $(\alpha, \alpha)$-edges is a sublinear $\alpha$-radial with root $r$. 
\end{lemma} 

\begin{proof} 
It is obvious that the obtained graph is an $\alpha$-radial with root $r$.  
Further, Lemma~\ref{lem:nonposi2slsr} implies that it is sublinear, because it has no $(-\alpha, -\alpha)$-edges. 
\end{proof}


Lemmas~\ref{lem:slr2decomp} and \ref{lem:const2lsr} now derive Theorem~\ref{thm:slr2char}. 

\begin{theorem} \label{thm:slr2char} 
Let $\alpha\in\{+, -\}$. 
A bidirected graph $G$ is a sublinear $\alpha$-radial with root $r\in V(G)$ 
if and only if it is a graph obtained by arbitrarily adding $(\alpha, \alpha)$-edges to an $\alpha$-flowgraph with root $r$. 
\end{theorem} 


Under Proposition~\ref{prop:flowgraph2char} and Theorem~\ref{thm:slr2char}, 
 the structure of sublinear radials are now understood from first principles.

\begin{ac} 
This study was supported by JSPS KAKENHI Grant Number 18K13451. 
\end{ac}

\bibliographystyle{splncs03.bst}
\bibliography{sear.bib}

\end{document}